\documentclass{amsart}
\usepackage[margin=1in]{geometry}
\usepackage{amsmath, amsfonts, amssymb, amsthm}
\usepackage{hyperref}
\usepackage{url}

\newtheorem{thm}{Theorem}
\newtheorem{lemma}[thm]{Lemma}

\newcommand{\R}{\mathbb R}
\newcommand{\eps}{\varepsilon}
\newcommand{\To}{\longrightarrow}
\newcommand{\dd}{\; \mathrm{d}}
\newcommand{\LL}{\mathcal{L}}

\title[Entropy production is not always monotone]{The entropy production is not always monotone\\ for hard spheres or for Maxwell molecules}

\author{Luis Silvestre}
\address{Mathematics Department, University of Chicago, Chicago, IL 60637, USA}
\email{luis@math.uchicago.edu}
\thanks{Luis Silvestre is supported by NSF grant DMS-2350263.}

\begin{document}

\begin{abstract}
In 1966, McKean asked whether the entropy production of the Boltzmann equation must be monotone decreasing in time. We show that this is not the case even in the space-homogeneous setting for 
the Boltzmann collision operator with a constant angular cross section and the kinetic parameter $\gamma \in [0,1]$. This recovers the classical case of hard spheres and the simplest case of Maxwell molecules.
Our examples are radially symmetric mixtures of Maxwellians.
\end{abstract}

\maketitle

\section{Introduction}
\label{s:setup}

In 1966, while studying the Kac caricature of a Maxwellian gas, McKean discovered that the Fisher information is monotone decreasing in time \cite{mckean1966speed}, and he conjectured that the entropy production should be monotone decreasing as well. He went further and suggested that all the successive derivatives of the entropy alternate in sign, a statement that became known as the \emph{super-$H$ theorem}. That conjecture attracted a lot of attention in the 1980s, but was eventually disproved at higher orders. Bobylev, and independently, Krook and Wu, found explicit solutions for Maxwell molecules \cite{bobylev1975exact,krook1976formation} whose derivatives were computed numerically and appeared to obey the expected signs
for roughly thirty orders \cite{ziff1981approach}; Olaussen then showed that the pattern actually fails at some high order, and Lieb gave another proof \cite{olaussen1982extension,lieb1982comment}. While these arguments prove that the super-$H$ theorem is false, they do not settle the question for the first few derivatives of the entropy. The first of these questions, which is in fact the first of McKean's conjectures in \cite{mckean1966speed}, is whether the entropy production is monotone decreasing in time. In a previous paper, the author proved it is not, for a collision kernel that is a singular measure concentrated on one relative speed and one scattering angle \cite{silvestre2026entropy}. This singular kernel allowed us to isolate the main culprit of the lack of monotonicity, but it left open the possibility that the physically motivated kernels behave better. In this paper we show that they do not. We present examples that show that the entropy production can fail to be decreasing for the space-homogeneous Boltzmann equation even in the case of Maxwell molecules or hard spheres.

We are interested in the space-homogeneous Boltzmann equation, in the form
\begin{equation} \label{e:boltzmann}
    \partial_t f = Q(f,f) ,
\end{equation}
where $Q(f,f)$ is the Boltzmann collision operator in 3D.
\[  Q(f,f) = \int_{\R^3} \int_{\mathbb{S}^2} (f' f'_* - f f_*)
    B(|v-v_*|, \cos \theta) \dd \sigma \dd v_* .\]

In this work, we will consider collision kernels $B$ that do not depend on the deviation angle $\theta$. That is, we set 
\begin{equation} \label{e:B}
    B = \frac{1}{4\pi} \, |v-v_*|^\gamma , \qquad \gamma \in [0,1] :
\end{equation}
the case $\gamma = 1$ corresponds to the hard-sphere model. The case $\gamma = 0$ is a Maxwell-molecule kernel with a constant angular cross section. We also consider the full range $\gamma \in [0,1]$ since the computation is essentially the same. The factor $1/(4\pi)$ is a normalization that does not affect the monotonicity result.

We analyze functions $f$ that solve \eqref{e:boltzmann}, are smooth, strictly positive, radial and rapidly decreasing. We construct them with unit mass. Their temperature is not set, but it can be adjusted by scaling a posteriori.

Given any function $f = f(t,v)$ solving the space-homogeneous Boltzmann equation, the entropy is monotone increasing in time. The entropy and the entropy production are computed directly by the formulas
\[ H(f) = -\int_{\R^3} f \log f \dd v, \qquad
   D(f) = \partial_t H(f) = -\int_{\R^3} Q(f,f) \log f \dd v \geq 0 . \]

The question we address in this note is whether the entropy production is monotone decreasing in time. That is, if $f$ solves \eqref{e:boltzmann}, whether
   \begin{equation} \label{e:question}
    \partial_t D(f) \leq 0 \, .
\end{equation}
We show that it is not always the case.

\begin{thm} \label{t:main}
Let $\gamma \in [0,1]$, let $p \in (0,p_\gamma)$, and for $R \geq 2$ let
$f_R = (1-p) M_1 + p M_R$, where $M_T(v) = (2\pi T)^{-3/2} e^{-|v|^2/(2T)}$.
Then, as $R \to \infty$,
\begin{equation} \label{e:thm}
    \partial_t D(f_R) = c_{p,\gamma} \, R^\gamma \log R
    + o_p \big( R^\gamma \log R \big) .
\end{equation}
In particular $\partial_t D(f_R) > 0$ for all $R$ large enough whenever
$0 < p < p_\gamma$. Here, $p_\gamma$ and $c_{p,\gamma}$ are positive numbers that can be computed explicitly.
\end{thm}

The constants are quite simple at the two endpoints of the range.
\begin{alignat*}{2}
    \partial_t D(f_R) &= \tfrac32 \, p(1-p)(1-2p) \log R + o_p(\log R)
    &\qquad& (\gamma = 0) , \\
    \partial_t D(f_R) &= \tfrac32 \, p(1-p)
    \Big( 1-p-\tfrac 8\pi p \Big) R \log R + o_p(R \log R)
    &\qquad& (\gamma = 1) ,
\end{alignat*}
so that $p_0 = \frac 12$ and $p_1 = \pi/(\pi+8) \approx 0.282$.

The densities $f_R$ are smooth, strictly positive, radial and rapidly
decreasing. Their temperature can be adjusted arbitrarily by scaling: if $S_\lambda f(v) = \lambda^3 f(\lambda v)$, then a kernel homogeneous of
degree $\gamma$ in the relative velocity satisfies
$Q(S_\lambda f, S_\lambda f) = \lambda^{-\gamma} S_\lambda Q(f,f)$, so that
\begin{equation} \label{e:scaling}
    \partial_t D(S_\lambda f) = \lambda^{-2\gamma} \, \partial_t D(f).
\end{equation}
Thus, the temperature can be made to be anything we want by choosing $\lambda > 0$ accordingly.

\subsection{Note on the Fisher information}
\label{s:biblio}

The monotonicity of the Fisher information has a more successful history than the monotonicity of the entropy production. Toscani and Villani established it for Maxwell molecules in the Boltzmann and Landau settings \cite{toscani1992new,villani1998fisher,villani2000decrease}. More recently, Guillen and the author proved that the Fisher information decreases for the Landau equation with a broad class of interaction potentials including the Coulomb case, and used it to rule out blow-up \cite{guillen2025landau}; see also the short exposition \cite{guillen2025fisher}. Imbert, Villani and the author found the corresponding criterion for the Boltzmann equation, in the form of an integral inequality on the sphere, and verified it for all the physical power-law interactions \cite{imbert2026monotonicity}; see also \cite{imbert2025fisher} and the survey \cite{villani2025fisher}. These ideas have been applied to prove the existence of global smooth solutions also starting from rough initial data \cite{ji2024dissipation,golding2024global}, to derive a weak solution to the Landau hierarchy \cite{carrillo2025fisher}, for a multi-species
version \cite{junne2025existence,zhu2025fisher,duong2026multispecies}, and in the inhomogeneous setting for the fuzzy Landau equation \cite{gualdani2025fuzzy}.

Collision kernels in Theorem \ref{t:main} belong to the class for which the
Fisher information is known to decrease \cite{imbert2026monotonicity}.

\subsection{Note on the Landau equation: an open problem}

The construction in this paper, as well as in the more transparent earlier construction by the author \cite{silvestre2026entropy}, is driven by collisions for which the deviation angle is not small. The Landau equation is the grazing collision limit of the Boltzmann equation. The ideas in this paper (or in \cite{silvestre2026entropy}) do not apply to the Landau equation.

Moreover, it is not hard to verify that the entropy production of the Landau equation with Maxwell molecules in the radially symmetric setting equals $ \frac 23 EI - 6M^2$, where $E$ is the energy, $M$ is the mass and $I$ is the usual Fisher information. As a consequence, the entropy production of the Landau equation is monotone decreasing in time for radially symmetric solutions. More generally, Tabary recently proved that the entropy production is nonincreasing whenever the directional temperatures are sufficiently well distributed and, in particular, after an explicit time for every finite-mass-and-energy initial datum \cite{tabary2026monotonicity}. Whether it is nonincreasing from the initial time for every solution remains open.

\subsection{Disclaimer on the use of AI tools}

A first version of the proof in this paper was obtained by ChatGPT 5.6 Sol. It was run in Codex, in Ultra mode, with access to the source of \cite{silvestre2026entropy} and other notes by the author. While the initial proof was not long, it was very difficult to read. It was rewritten first using Claude code, and then reinterpreted and restructured by the author. The author takes full responsibility for the final version of the proof.

\subsection{Notation}

We use the following notation that is standard in the analysis of kinetic equations.

For a function $h$ on $\R^3$ we write $h_* = h(v_*)$, $h' = h(v')$ and
$h'_* = h(v'_*)$, where
\begin{equation} \label{e:vprime}
    v' = \frac{v+v_*}{2} + \frac{|v-v_*|}{2} \sigma, \qquad
    v'_* = \frac{v+v_*}{2} - \frac{|v-v_*|}{2} \sigma .
\end{equation}

The collision operator is written as a bilinear operator using the following convention
\begin{equation} \label{e:Q}
    Q(f,g)(v) = \int_{\R^3} \int_{\mathbb{S}^{2}}
    \big( f'_* g' - f_* g \big) B \dd \sigma \dd v_* .
\end{equation}

The subscript in $o_p$, $O_p$ and $C_p$ means dependence on the mixture weight
$p$, which is fixed once and for all, but never on the temperature ratio $R$,
which is the large parameter; $C_p$ may change from line to line.

\section{Setting up the computation}
\label{s:twoscales}

We recall that, by the standard pre-post collisional change of variables applied
to the gain half of \eqref{e:Q}, the following formula holds:
\begin{equation} \label{e:weak}
    \int_{\R^3} Q(f,g) \varphi \dd v
    = \int_{\R^3}\int_{\R^3} \int_{\mathbb{S}^{2}}
    f_* \, g \, \big( \varphi' - \varphi \big) B \dd \sigma \dd v_* \dd v .
\end{equation}

The following identity was derived by the author in \cite{silvestre2026entropy}. It is the starting point of our analysis here as well.
\begin{equation} \label{e:start}
    \partial_t D(f) = -\int_{\R^3} \frac{Q^2}{f} \dd v
    + \int_{\R^3} Q \, \LL \dd v ,
    \qquad
    \LL(v) = \int_{\R^3} \int_{\mathbb{S}^{2}} f_*
    \log \left( \frac{f f_*}{f' f'_*} \right) B \dd \sigma \dd v_* .
\end{equation}
The first term is negative and quadratic in $Q$. The second term has no sign. Note that the integrands of $Q$ and $\LL$ have pointwise opposite signs. For any values of $v$, $v_*$, $v'$, $v'_*$,
\[ (f'_* f' - f_* f) \text{ has the opposite sign as } f_* \log \left( \frac{f f_*}{f' f'_*} \right).\]
Thus, the values of $Q$ and $\LL$ will typically have the opposite sign but not always. Our challenge is to make $Q$ and $\LL$ have the same sign at sufficiently many points $v$ so as to make the second term positive and an order of magnitude larger than the first term. We will achieve this by a $\log R$ factor, resembling the construction in \cite{silvestre2026entropy}.

Since Maxwellians are at equilibrium, we have $Q(M_1,M_1) = Q(M_R,M_R) = 0$. Therefore, by the bilinearity of $Q$,
\begin{equation} \label{e:A}
    Q(f_R,f_R) = 2p(1-p) \, A_R , \qquad
    A_R := \tfrac 12 \big( Q(M_1,M_R) + Q(M_R,M_1) \big) .
\end{equation}
We may interpret $Q(M_1,M_R) + Q(M_R,M_1)$ as a single cold--hot interaction.
Replacing in \eqref{e:start}, this is the formula we shall evaluate:
\begin{equation} \label{e:master}
    \partial_t D(f_R) = -4p^2(1-p)^2 \int_{\R^3} \frac{A_R^2}{f_R} \dd v
    + 2p(1-p) \int_{\R^3} A_R \, \LL \dd v .
\end{equation}
In Section \ref{s:negative} we will show that the first negative term is of order $O_p(R^\gamma)$. In Section \ref{s:LL}, we will estimate that the second is asymptotically $\approx R^\gamma \log R$, with a positive constant factor when $p$ is small enough. The second term therefore beats the first by a factor
$\log R$, and decides the sign.

\medskip

Some of the terms in the collision operator can be visualized in terms of interactions between cold and hot particles. We say that a velocity is \emph{cold} if it is of order one and \emph{hot} if it is of order $\sqrt R$. We see how the two scales interact through the pre-post collisional identities \eqref{e:vprime}. If $v$ is cold and $v_*$ is hot, then $|v-v_*|$ is most often of size
$\sqrt R$ (unless the collision is grazing), so both outgoing velocities are hot. When we take limits as $R \to \infty$, we observe the following: if $v_R \to u$ and $v_{R*}/\sqrt R \to z \neq 0$, then
\begin{equation} \label{e:hotout}
    \frac{v'_R}{\sqrt R} \To y_+ := \frac{z+|z|\sigma}{2},
    \qquad
    \frac{v'_{R*}}{\sqrt R} \To y_- := \frac{z-|z|\sigma}{2},
    \qquad
    |y_\pm| = |z| \sqrt{\frac{1 \pm t}{2}},
\end{equation}
where $t = \widehat z \cdot \sigma$, and $y_\pm \neq 0$ for almost every
$\sigma$. A hot--hot collision behaves the same way: if $v_R/\sqrt R \to z$ and
$v_{R*}/\sqrt R \to \zeta$ with $z \neq \zeta$, then
\begin{equation} \label{e:hothot}
    \frac{v'_R}{\sqrt R} \To \frac{z+\zeta}{2} + \frac{|z-\zeta|}{2} \sigma ,
\end{equation}
which is again nonzero for almost every $\sigma$, and likewise for $v'_{R*}$.
Finally $|v'| \leq |v| + |v_*|$, so a cold--cold collision has two cold outputs.
In short, a collision consumes a cold particle when one of its two
inputs is cold and the other is hot.

In $\LL$, we see the logarithm of the ratio of the products of the incoming and outgoing densities. This logarithm would vanish pointwise if $f$ were a single Maxwellian. To study that quantity, we introduce the function $\psi_R(v)$ below. Set $f_R = p M_R \big( 1 + \frac{1-p}{p} M_1/M_R \big)$,
which gives the exact formula
\begin{equation} \label{e:psi}
    \log f_R(v) = \log \frac{p}{(2\pi R)^{3/2}} - \frac{|v|^2}{2R} + \psi_R(v) ,
    \qquad
    \psi_R(v) := \log \left( 1 + \frac{1-p}{p} R^{3/2} e^{-\theta_R |v|^2}
    \right) ,
\end{equation}
where $\theta_R = \frac{R-1}{2R} \in \big[ \frac 14, \frac 12 \big]$. The function $\psi_R(v)$ captures the part of $f_R$ that contributes to the logarithm in $\mathcal L$. Indeed, the four logarithms in \eqref{e:start} occur with two plus
signs and two minus signs, so the constant cancels, and $|v|^2/(2R)$ cancels as
well by the conservation of energy $|v'|^2+|v'_*|^2 = |v|^2+|v_*|^2$. Hence
\begin{equation} \label{e:Lpsi}
    \log \frac{f_R f_{R*}}{f'_R f'_{R*}} = -\Delta \psi_R ,
    \qquad
    \Delta \psi_R := \psi'_R + \psi'_{R*} - \psi_R - \psi_{R*} ,
\end{equation}
and therefore
\begin{equation} \label{e:LLpsi}
    \LL(v) = -\int_{\R^3} \int_{\mathbb{S}^2} f_R(v_*) \, \Delta \psi_R \,
    |v-v_*|^\gamma \frac{\dd \sigma}{4\pi} \dd v_* .
\end{equation}

Understanding the asymptotic behavior of $\Delta \psi_R$ as $R \to \infty$ is key to the proof of Theorem \ref{t:main}. We start with a simple upper bound.

\begin{lemma} \label{l:psi}
For every $p \in (0,1)$,
\begin{equation} \label{e:psi-bound}
    0 \leq \psi_R(v) \leq \tfrac 32 \log R + O_p(1)
    \qquad \text{for all } v \in \R^3 \text{ and } R \geq 2 .
\end{equation}
Moreover, if we let $R \to \infty$ and $v_R \to v$ then $\psi_R(v_R) = \frac 32 \log R + O_p(1)$, and if
$v_R/\sqrt R \to y \neq 0$ then $\psi_R(v_R) \to 0$.
\end{lemma}

\begin{proof}
All three statements are read off the definition of $\psi_R$ in \eqref{e:psi}.
It is nonnegative, and at most
$\log \big( 1 + \frac{1-p}{p} R^{3/2} \big) = \frac 32 \log R + O_p(1)$. If
$v_R \to v$ then
$\theta_R |v_R|^2$ is bounded, so $R^{3/2} e^{-\theta_R |v_R|^2} \to \infty$
like $R^{3/2}$ and $\psi_R(v_R) = \frac 32 \log R + O_p(1)$. If
$v_R = \sqrt R \, y_R$ with
$y_R \to y \neq 0$, then $R^{3/2} e^{-\theta_R R |y_R|^2} \to 0$ because
$\theta_R \geq 1/4$, and $\psi_R(v_R) \to 0$.
\end{proof}

Combining Lemma \ref{l:psi} with \eqref{e:hotout}, \eqref{e:hothot} and
$|v'| \leq |v|+|v_*|$, we observe the following dichotomy.
\begin{equation} \label{e:dichotomy}
    -\frac{\Delta \psi_R}{\log R} \To
    \begin{cases}
        \tfrac 32 & \text{for a cold--hot collision}, \\
        0 & \text{for a cold--cold or a hot--hot collision},
    \end{cases}
\end{equation}
for almost every $\sigma$. Here the labels refer to the two incoming
velocities: \emph{cold} means a sequence converging in $\R^3$, \emph{hot} means
a sequence whose quotient by $\sqrt R$ converges to a nonzero limit, and in the
hot--hot case the two limits are distinct.

\section{The negative term}
\label{s:negative}

The purpose of this section is to prove the following lemma describing the asymptotic behavior of the first term in \eqref{e:master}.

\begin{lemma} \label{l:chi}
For every $p \in (0,1)$, $\displaystyle
\int_{\R^3} \frac{A_R^2}{f_R} \dd v = O_p(R^\gamma)$ as $R \to \infty$.
\end{lemma}

Let us define the $s$th moments of the unit Maxwellian $M_1$ as $\kappa_s$. These quantities will come up in some computations in this paper. For $s \geq 0$,
\begin{equation} \label{e:kappa}
    \kappa_s = \int_{\R^3} |z|^s M_1(z) \dd z
    = 2^{s/2} \, \frac{\Gamma\big( \tfrac{3+s}{2} \big)}
                      {\Gamma\big( \tfrac 32 \big)} ,
    \qquad
    \kappa_0 = 1, \quad \kappa_1 = 2\sqrt{\tfrac 2\pi}, \quad \kappa_2 = 3 ,
\end{equation}
so that $\kappa_1^2 = 8/\pi$. Furthermore, for $T > 0$
\begin{equation} \label{e:nu}
    \nu_T(v) = \int_{\R^3} |v-v_*|^\gamma M_T(v_*) \dd v_*
\end{equation}
is the frequency at which a particle at $v$ collides with a Maxwellian of
temperature $T$. Since $\gamma \in [0,1]$, the function $s \mapsto s^\gamma$ is
subadditive, so that $|v-v_*|^\gamma \leq \big( |v| + |v_*| \big)^\gamma \leq
|v|^\gamma + |v_*|^\gamma$. We thus obtain the following elementary upper bound
for $\nu_T(v)$,
\begin{align}
\nu_T(v) &\leq \int_{\R^3} \left( |v|^\gamma + |v_*|^\gamma \right) M_T(v_*) \dd v_* \nonumber \\
&= |v|^\gamma + \kappa_\gamma T^{\gamma/2} . \label{e:nu-bound}
\end{align}

For any two probability densities $f$ and $g$, we define $P(f,g)$ as follows:
\begin{equation} \label{e:Pdef}
    P(f,g)(v) = \int_{\R^3} \int_{\mathbb{S}^2} f(v'_*) \, g(v')
    \frac{\dd \sigma}{4\pi} \dd v_* ,
\end{equation}
which is the gain part of \eqref{e:Q} when $\gamma = 0$. Note that we use this same definition for $P(f,g)$ regardless of the value of $\gamma$.

Using \eqref{e:weak} turns \eqref{e:Pdef} into
\begin{equation} \label{e:P}
    \int_{\R^3} P(f,g) \varphi \dd v
    = \int_{\R^3}\int_{\R^3} f(v) g(v_*) \int_{\mathbb{S}^2} \varphi(v')
      \frac{\dd \sigma}{4\pi} \dd v_* \dd v
\end{equation}
for every bounded continuous $\varphi$. It is easy to see that $P(f,g)$ is again a probability density and it is symmetric in $f$ and $g$.

The following lemma could be derived easily using Bobylev's Fourier representation of the Maxwell-molecule collision operator \cite{bobylev1975fourier}. We provide a self-contained direct computation in physical variables.

\begin{lemma} \label{l:mixture}
For $a, b > 0$ with $a \neq b$,
\[ P(M_a,M_b) = \frac 12 \int_{-1}^1 M_{\theta(t)} \dd t
   = \frac{1}{|b-a|} \int_{\min\{a,b\}}^{\max\{a,b\}} M_u \dd u ,
   \qquad \theta(t) = \frac{a+b}{2} + t \, \frac{a-b}{2} .\]
\end{lemma}

\begin{proof}
By \eqref{e:P} it is enough to show that
\[ I := \int_{\R^3}\int_{\R^3} M_a(v) M_b(v_*) \int_{\mathbb{S}^2} \varphi(v')
   \frac{\dd \sigma}{4\pi} \dd v_* \dd v
   = \frac 12 \int_{-1}^1 \int_{\R^3} M_{\theta(t)} \varphi \dd v \dd t \]
for every bounded continuous $\varphi$. Pass to $m = \frac{v+v_*}{2}$ and
$z = \frac{v-v_*}{2}$, so that $v' = m+|z|\sigma$ and
$\dd v \dd v_* = 8 \dd m \dd z$. With $\lambda = (a-b)/(a+b)$, completing the
square in $m$ gives
\[ \frac{|m+z|^2}{2a} + \frac{|m-z|^2}{2b}
   = \frac{a+b}{2ab} \big| m - \lambda z \big|^2 + \frac{2|z|^2}{a+b} ,
   \qquad \text{that is} \qquad
   8 M_a(v) M_b(v_*) = M_\alpha(m-\lambda z) \, M_\beta(z) , \]
where $\alpha = ab/(a+b)$ and $\beta = (a+b)/4$. Substituting $m = g+\lambda z$
and writing $\widehat z = z/|z|$,
\[ I = \int_{\R^3}\int_{\R^3} M_\alpha(g) \, M_\beta(z) \int_{\mathbb{S}^2}
   \varphi \big( g + |z| (\lambda \widehat z + \sigma) \big)
   \frac{\dd \sigma}{4\pi} \dd z \dd g . \]
Since $M_\beta$ is radial, the $z$-integral is unchanged if the inner integral
is replaced by its average over $\widehat z \in \mathbb{S}^2$. The measure
$\frac{\dd \widehat z}{4\pi} \frac{\dd \sigma}{4\pi}$ so obtained is invariant
under simultaneous rotations, which fix $t = \widehat z \cdot \sigma$, and
$|\lambda \widehat z + \sigma|^2 = 1+\lambda^2+2\lambda t$; averaging over the
rotation group therefore replaces $\lambda \widehat z + \sigma$ by
$\sqrt{1+\lambda^2+2\lambda t} \, \eta$ with $\eta$ averaged over
$\mathbb{S}^2$, and leaves $t$ the measure $\frac 12 \dd t$ on $[-1,1]$. Undoing
that average by radiality once more and substituting
$\zeta = \sqrt{1+\lambda^2+2\lambda t} \, z$,
\[ I = \frac 12 \int_{-1}^1 \int_{\R^3}\int_{\R^3} M_\alpha(g) \,
   M_{(1+\lambda^2+2\lambda t)\beta}(\zeta) \, \varphi(g+\zeta)
   \dd \zeta \dd g \dd t . \]
Finally $M_s * M_{s'} = M_{s+s'}$, and $1+\lambda^2 = 2(a^2+b^2)/(a+b)^2$ gives
\[ \alpha + \big( 1+\lambda^2+2\lambda t \big) \beta
   = \frac{a+b}{2} + t \, \frac{a-b}{2} = \theta(t) , \]
which is the first equality.
\end{proof}

\begin{lemma} \label{l:L2}
For every $a \in (0,2)$ there is a constant $C_a$ such that
\[ \int_{\R^3} \frac{P(M_a, M_{aR})^2}{M_R} \dd v \leq C_a
   \qquad \text{for all } R \geq 2 . \]
\end{lemma}

\begin{proof}
A direct Gaussian computation gives, for $0 < u < 2R$,
\[ \int_{\R^3} \frac{M_u^2}{M_R} \dd v = \big( x(2-x) \big)^{-3/2},
   \qquad x = \frac uR . \]
By Lemma \ref{l:mixture} and Minkowski's integral inequality,
\[ \left( \int_{\R^3} \frac{P(M_a,M_{aR})^2}{M_R} \dd v \right)^{1/2}
   \leq \frac{1}{a(R-1)} \int_a^{aR}
   \left( \frac uR \Big( 2 - \frac uR \Big) \right)^{-3/4} \dd u
   = \frac{R}{a(R-1)} \int_{a/R}^{a} \big( x(2-x) \big)^{-3/4} \dd x . \]
On $(0,a]$ we have $2-x \geq 2-a > 0$, so the last integral is at most
$(2-a)^{-3/4} \int_0^a x^{-3/4} \dd x = 4 a^{1/4} (2-a)^{-3/4}$, and
$R/(R-1) \leq 2$ for $R \geq 2$.
\end{proof}

\begin{proof}[Proof of Lemma \ref{l:chi}]
Since $\int_{\mathbb{S}^2} B \dd \sigma = |v-v_*|^\gamma$, the loss part of
\eqref{e:Q} is explicit in terms of $\nu_R$ and $\nu_1$. Thus, $A_R$ splits as
\begin{equation} \label{e:split}
    A_R = G_R - \tfrac 12 M_1 \nu_R - \tfrac 12 M_R \nu_1 ,
\end{equation}
where $G_R \geq 0$ collects the two gain terms. By \eqref{e:weak}, and after
exchanging $v$ and $v_*$ in the one coming from $Q(M_1,M_R)$,
\begin{equation} \label{e:gain}
    \int_{\R^3} G_R \, \varphi \dd v = \int_{\R^3}\int_{\R^3} M_1(v) M_R(v_*)
    |v-v_*|^\gamma \int_{\mathbb{S}^2} \varphi(v') \frac{\dd \sigma}{4\pi}
    \dd v_* \dd v .
\end{equation}
Using $f_R \geq (1-p) M_1$, $f_R \geq p M_R$ and the bound on $\nu_T$ in
\eqref{e:nu-bound},
\[ \int_{\R^3} \frac{(M_1 \nu_R)^2}{f_R} \leq
   \frac 1{1-p} \int_{\R^3} M_1 \big( |v|^{\gamma} + \kappa_\gamma R^{\gamma/2}\big)^2
   = O_p(R^\gamma),
   \quad
   \int_{\R^3} \frac{(M_R \nu_1)^2}{f_R} \leq
   \frac 1p \int_{\R^3} M_R \big( |v|^\gamma + \kappa_\gamma \big)^{2}
   = O_p(R^\gamma). \]
Both estimates follow by expanding the square and using
$\int_{\R^3} M_T(w) |w|^s \dd w = T^{s/2} \kappa_s$.

For the gain term, fix $a = 1+\eps$ with $\eps$ small, so that $a < 2$. We claim
that
\begin{equation} \label{e:absorption}
    |v-v_*|^\gamma M_1(v) M_R(v_*) \leq C_\eps R^{\gamma/2} M_a(v) M_{aR}(v_*)
\end{equation}
uniformly for $R \geq 2$. Indeed, after dividing by the right hand side, and
writing $s = |v|$ and $\tau = |v_*|/\sqrt R$, the nonconstant factor is at most a
constant multiple of $(s+\tau)^\gamma e^{-\frac{\eps}{2a}(s^2+\tau^2)}$, which is
bounded. The role of \eqref{e:absorption} is to remove the factor
$|v-v_*|^\gamma$ from \eqref{e:gain} at the cost of $R^{\gamma/2}$; what is left
is \eqref{e:P} with $f = M_a$ and $g = M_{aR}$, so that
$G_R \leq C_\eps R^{\gamma/2} P(M_a,M_{aR})$ for every $\gamma \in [0,1]$. By
Lemma \ref{l:L2},
\[ \int_{\R^3} \frac{G_R^2}{f_R} \dd v
   \leq \frac 1p \int_{\R^3} \frac{G_R^2}{M_R} \dd v
   \leq \frac{C_\eps^2 R^\gamma}{p}
   \int_{\R^3} \frac{P(M_a,M_{aR})^2}{M_R} \dd v = O_p(R^\gamma) . \qedhere \]
\end{proof}

\section{The second term}
\label{s:LL}

In this section we analyze the behavior of the second term in \eqref{e:start} (which is the same as the second term in \eqref{e:master}) as $R \to \infty$. The first lemma establishes an upper bound and studies the pointwise asymptotics of $\LL(v)$ as $R \to \infty$.

\begin{lemma} \label{l:LL}
There is a constant $C_p$ such that
\begin{equation} \label{e:LL-bound}
    |\LL(v)| \leq C_p R^{\gamma/2}
    \left( 1 + \frac{|v|^2}{R} \right)^{\gamma/2} \log R
    \qquad \text{for all } v \in \R^3 \text{ and } R \geq 2 .
\end{equation}
Moreover $\LL$, like $f_R$, depends on $R$, and as $R \to \infty$ the following
hold: if $v_R \to u$ then
$\LL(v_R)/(R^{\gamma/2} \log R) \to \frac 32 p \, \kappa_\gamma$, and if
$v_R/\sqrt R \to y \neq 0$ then
$\LL(v_R)/(R^{\gamma/2} \log R) \to \frac 32 (1-p) |y|^\gamma$.
\end{lemma}

\begin{proof}
Splitting $f_R$ into its two components, \eqref{e:LLpsi} reads
\begin{equation} \label{e:LL-split}
    \LL(v) = (1-p) \int_{\R^3} \int_{\mathbb{S}^2} M_1(u) \, |v-u|^\gamma
    \big( -\Delta \psi_R \big) \frac{\dd \sigma}{4\pi} \dd u
    + p \int_{\R^3} \int_{\mathbb{S}^2} M_R(u) \, |v-u|^\gamma
    \big( -\Delta \psi_R \big) \frac{\dd \sigma}{4\pi} \dd u ,
\end{equation}
where $\Delta \psi_R$ is computed from $v$, $u$ and $\sigma$ by
\eqref{e:vprime}.

For \eqref{e:LL-bound}, we use the bound \eqref{e:psi-bound} which gives
$|\Delta \psi_R| \leq 4 \big( \tfrac 32 \log R + O_p(1) \big) \leq C_p \log R$
regardless of the values of $v$, $v_*$, $v'$ and $v'_*$. Thus, \eqref{e:LL-split},
\eqref{e:nu} and \eqref{e:nu-bound} give
\[ |\LL(v)| \leq C_p \big( (1-p) \nu_1(v) + p \, \nu_R(v) \big) \log R
   \leq C_p \big( |v|^\gamma + \kappa_\gamma R^{\gamma/2} \big) \log R
   \leq C_p R^{\gamma/2}
   \left( 1 + \frac{|v|^2}{R} \right)^{\gamma/2} \log R .\]
In the second inequality we applied \eqref{e:nu-bound} with $T = 1$ and $T = R$,
together with $(1-p) + p R^{\gamma/2} \leq R^{\gamma/2}$; in the last inequality, we used that
$|v|^\gamma$ and $R^{\gamma/2}$ are both at most $\big( R + |v|^2 \big)^{\gamma/2}$.

For the two limits, substitute $u = \sqrt R \, w$ in the second integral of
\eqref{e:LL-split}, so that $M_R(u) \dd u = M_1(w) \dd w$, and divide by
$R^{\gamma/2} \log R$:
\[ \frac{\LL(v)}{R^{\gamma/2} \log R}
   = (1-p) \int\!\!\int M_1(u) \frac{|v-u|^\gamma}{R^{\gamma/2}}
     \frac{-\Delta \psi_R}{\log R} \frac{\dd \sigma}{4\pi} \dd u
   + p \int\!\!\int M_1(w) \Big| \frac{v}{\sqrt R} - w \Big|^\gamma
     \frac{-\Delta \psi_R}{\log R} \frac{\dd \sigma}{4\pi} \dd w . \]
Since $|\Delta \psi_R| \leq C_p \log R$, the two integrands are dominated by
constant multiples of $M_1(u) (1+|u|)^\gamma$ and $M_1(w) (1+|w|)^\gamma$, so
dominated convergence applies to both. If $v_R \to u_0$, the first collision is
cold--cold and $-\Delta \psi_R/\log R \to 0$ by \eqref{e:dichotomy}, so that
term
vanishes; in the second $|v_R/\sqrt R - w|^\gamma \to |w|^\gamma$ and
$-\Delta \psi_R/\log R \to \frac 32$, and the limit is
$\frac 32 p \int M_1 |w|^\gamma \dd w = \frac 32 p \kappa_\gamma$. If instead
$v_R/\sqrt R \to y \neq 0$, then $|v_R-u|^\gamma/R^{\gamma/2} \to |y|^\gamma$ and
the first collision is hot--cold, so that term tends to
$\frac 32 (1-p)|y|^\gamma$, while
the second is hot--hot, with distinct limits $y \neq w$ for almost every $w$,
and vanishes.
\end{proof}

In words: $\LL \approx \frac 32 p \, \kappa_\gamma R^{\gamma/2} \log R$ on the
cold scale and $\LL \approx \frac 32 (1-p) |v|^\gamma \log R$ on the hot scale. In the next lemma, we analyze how these asymptotics balance in the second term of \eqref{e:master} when we multiply $\LL$ by $A_R$ and integrate.
%A cold particle consumes cold mass only when it meets one of the $p$ hot ones, and a hot particle only when it meets one of the $1-p$ cold ones, each meeting weighted by the frequency at which it occurs.

\begin{lemma} \label{l:second}
As $R \to \infty$,
\[ \int_{\R^3} A_R \, \LL \dd v
   = \frac 34 \left( (1-p) \, \frac{2-\gamma}{2+\gamma} \, \kappa_{2\gamma}
   - p \, \kappa_\gamma^2 \right) R^\gamma \log R
   + o_p \big( R^\gamma \log R \big) . \]
\end{lemma}

\begin{proof}
Since $A_R = \frac 12 \big( Q(M_1,M_R) + Q(M_R,M_1) \big)$, applying
\eqref{e:weak} with $\varphi = \LL$ to each of the two terms and symmetrizing in
$v$, $v_*$, $v'$ and $v'_*$ in the standard way gives
\[ \int_{\R^3} A_R \LL \dd v = \frac 12 \int_{\R^3} \int_{\R^3}
   \int_{\mathbb{S}^2} M_1(v) M_R(v_*)
   \big( \LL' + \LL'_* - \LL - \LL_* \big) B \dd \sigma \dd v_* \dd v . \]
Now $B = |v-v_*|^\gamma/(4\pi)$, and the substitution $v_* = \sqrt R \, w$ gives
$M_R(v_*) \dd v_* = M_1(w) \dd w$ and
$|v-v_*|^\gamma = R^{\gamma/2} \big| w - v/\sqrt R \big|^\gamma$. Dividing by
$R^\gamma \log R$,
\begin{equation} \label{e:weakA}
    \frac{1}{R^\gamma \log R} \int_{\R^3} A_R \LL \dd v
    = \frac 12 \int_{\R^3} \int_{\R^3} M_1(v) M_1(w)
    \Big| w - \frac{v}{\sqrt R} \Big|^\gamma \int_{\mathbb{S}^2}
    \frac{\LL' + \LL'_* - \LL(v) - \LL(\sqrt R w)}{R^{\gamma/2} \log R}
    \frac{\dd \sigma}{4\pi} \dd w \dd v ,
\end{equation}
where $v'$ and $v'_*$ are computed from $v$ and $\sqrt R w$ by
\eqref{e:vprime}.

Fix $v$, $w \neq 0$ and $\sigma$, and let $R \to \infty$. The velocity $v$ is on
the cold scale and $\sqrt R w$ is on the hot scale, and \eqref{e:hotout} applies
with $z = w$: the outgoing velocities satisfy $v'/\sqrt R \to y_+$ and
$v'_*/\sqrt R \to y_-$, with $|y_\pm| = |w| \big( \frac{1 \pm t}{2} \big)^{1/2}$
and $t = \widehat w \cdot \sigma$, and $y_\pm \neq 0$ for almost every $\sigma$.
The four limits of Lemma \ref{l:LL} therefore make the integrand of
\eqref{e:weakA} converge to
\[ \frac 32 \, M_1(v) M_1(w) |w|^\gamma
   \Big[ (1-p) \big( |y_+|^\gamma + |y_-|^\gamma \big)
   - p \, \kappa_\gamma - (1-p) |w|^\gamma \Big] . \]
Moreover $|w - v/\sqrt R|^\gamma \leq (|v|+|w|)^\gamma$, and \eqref{e:LL-bound}
together with $|v'|^2 + |v'_*|^2 = |v|^2 + R|w|^2$ bounds each of the four
quotients by $C_p (1+|v|^2+|w|^2)^{\gamma/2}$, so the integrand is dominated by
an integrable function of $(v,w)$. As the limit does not depend on $v$ and
$\int M_1 = 1$, dominated convergence gives
\begin{equation} \label{e:limit}
    \frac{1}{R^\gamma \log R} \int_{\R^3} A_R \LL \dd v \To
    \frac 34 \int_{\R^3} M_1(w) |w|^\gamma \int_{\mathbb{S}^2}
    \Big[ (1-p) \big( |y_+|^\gamma + |y_-|^\gamma \big)
    - p \, \kappa_\gamma - (1-p) |w|^\gamma \Big]
    \frac{\dd \sigma}{4\pi} \dd w .
\end{equation}
Two of the three integrals in \eqref{e:limit} do not involve $\sigma$; by the
definition \eqref{e:kappa} of $\kappa_s$ they equal $p \kappa_\gamma^2$ and
$(1-p) \kappa_{2\gamma}$. In the third,
$|y_\pm|^\gamma = |w|^\gamma \big( \tfrac{1 \pm t}{2} \big)^{\gamma/2}$. Since
$\int_{\mathbb{S}^2} g(\widehat w \cdot \sigma) \, \tfrac{\dd \sigma}{4\pi}
= \tfrac 12 \int_{-1}^1 g \dd t$ for every $g$,
\[ \int_{\mathbb{S}^2} \Big( \frac{1 \pm t}{2} \Big)^{\gamma/2}
   \frac{\dd \sigma}{4\pi}
   = \frac 12 \int_{-1}^1 \Big( \frac{1 \pm t}{2} \Big)^{\gamma/2} \dd t
   = \int_0^1 u^{\gamma/2} \dd u = \frac{2}{2+\gamma} , \]
so that the third integral is $(1-p) \frac{4}{2+\gamma} \kappa_{2\gamma}$.
Altogether the right hand side of \eqref{e:limit} equals
\[ \frac 34 \left[ (1-p) \kappa_{2\gamma}
   \left( \frac{4}{2+\gamma} - 1 \right) - p \, \kappa_\gamma^2 \right]
   = \frac 34 \left( (1-p) \, \frac{2-\gamma}{2+\gamma} \, \kappa_{2\gamma}
   - p \, \kappa_\gamma^2 \right) . \qedhere \]
\end{proof}

\section{Proof of Theorem \ref{t:main}}
\label{s:proof}

It remains to add the two estimates. By Lemma \ref{l:chi} the first term of
\eqref{e:master} is $O_p(R^\gamma) = o_p(R^\gamma \log R)$, and by
Lemma \ref{l:second} the second one is $2p(1-p)$ times the quantity computed
there. Their sum is \eqref{e:thm}, with
\[ c_{p,\gamma} = \frac 32 \, p (1-p) \left( (1-p) \, \frac{2-\gamma}{2+\gamma}
   \, \kappa_{2\gamma} - p \, \kappa_\gamma^2 \right) ,
   \qquad
   p_\gamma = \left( 1 + \frac{(2+\gamma) \, \kappa_\gamma^2}
                             {(2-\gamma) \, \kappa_{2\gamma}} \right)^{-1} , \]
the second formula being the rearrangement of $c_{p,\gamma} > 0$. \qed

\bibliographystyle{plain}
\bibliography{entropy-production}

\end{document}